\documentclass[12pt,letterpaper]{amsart}
\makeatletter
\@namedef{subjclassname@2020}{%
\textup{2020} Mathematics Subject Classification}
\makeatother
\usepackage{geometry,xcolor,tikz}
\usepackage{mathrsfs,latexsym,url,setspace}
\usepackage{mathpazo} %Fonts
\usepackage[colorlinks,allcolors=blue]{hyperref} %pagebackref=true
\usepackage{xpatch} 
\xpatchcmd{\proof} 
{\itshape} 
{\bfseries} 
{}
{}
\newtheorem{theorem}{Theorem}
\newtheorem{theoremA}{Theorem}

\newtheorem*{theorem*}{Theorem*}
\newtheorem{proposition}{Proposition}
\newtheorem{lemma}{Lemma}
\newtheorem{corollary}{Corollary}
\theoremstyle{remark}

\newtheorem{example}{Example}
\newcommand{\C}{\mathbb{C}}
\newcommand{\disk}{\mathbb{D}}
\newcommand{\D}{\Omega}
\newcommand{\ep}{\varepsilon}
\newcommand{\zb}{\overline{z}}
\usepackage{soul}
\onehalfspace
\allowdisplaybreaks %breaking equations across pages
\title{Compactness of composition operators on the Bergman space of the bidisc}

\author{Timothy G. Clos}
\address[Timothy G. Clos]{Kent State University, Department of Mathematical 
	Sciences, Kent, OH 44242, USA}
\email{tclos@kent.edu}

\author{\v{Z}eljko \v{C}u\v{c}kovi\'c}
\address[\v{Z}eljko \v{C}u\v{c}kovi\'c]{University of Toledo, 
Department of Mathematics \& Statistics, Toledo, OH 43606, USA}
\email{Zeljko.Cuckovic@utoledo.edu}

\author{S\"{o}nmez \c{S}ahuto\u{g}lu}
\address[S\"{o}nmez \c{S}ahuto\u{g}lu]{University of Toledo, 
	Department of Mathematics \& Statistics, Toledo, OH 43606, USA}
\email{sonmez.sahutoglu@utoledo.edu}

\date{\today}
\begin{document}

\begin{abstract}
Let $\varphi$ be a holomorphic self-map of the bidisc that is Lipschitz 
on the closure. We show that the composition operator $C_{\varphi}$ 
is compact on the Bergman space if and only if 
$\varphi(\overline{\mathbb{D}^2})\cap \mathbb{T}^2=\emptyset$ 
and  $\varphi(\overline{\mathbb{D}^2}\setminus \mathbb{T}^2)
\cap b\mathbb{D}^2=\emptyset$. In the last section of the paper, 
we prove a result on $C^2$-smooth bounded pseudoconvex 
domains in $\mathbb{C}^n$. 
\end{abstract}
\maketitle
%%%%%%%%%%%%%%%%%%%%%%%%%%

Let $\Omega$ be a domain in $\mathbb{C}^n$ and suppose 
$\varphi:\Omega\rightarrow \Omega$ is holomorphic.  Then the composition 
operator with symbol $\varphi$, acting on the space of all holomorphic 
functions on $\Omega$, is defined by 
\[C_{\varphi}f(z)=f\circ \varphi(z)\] 
for $z\in \Omega$. Composition operators have been studied extensively 
by many authors on different spaces of holomorphic functions and on various 
domains.  Some of the more common spaces studied in one complex dimension 
include the Hardy and Bergman spaces on the unit disc, and the Segal-Bargmann 
space on $\mathbb{C}$.  Sometimes weighted versions of these spaces are 
studied with various weights.  In several complex dimensions, the common 
spaces  studied are the Hardy and Bergman spaces of the polydisc, and 
the Bergman space of the unit ball.  

In this paper we are interested in studying compactness of composition 
operators on the Bergman space on the bidisc $\mathbb{D}^2$, 
where $\mathbb{D}$ is the open unit disc in $\mathbb{C}$.  
We expect the higher dimensional results to be more complicated.  

Recall that the Bergman space $A^2(\mathbb{D}^n)$ is the space of all 
holomorphic functions $f: \mathbb{D}^n \rightarrow \mathbb{C}$ for which
\[\int_{\mathbb{D}^n} |f(z)|^2 dV(z) <\infty\]
where $dV$ denotes the volume measure on $\disk^n$. 
It is well known that the point evaluation functional is linear and 
bounded on the Bergman space on any domain $\Omega\subset \mathbb{C}^n$ 
and for any $n\in \mathbb{N}$.  Therefore, by the Riesz representation 
theorem, there exists $K_a\in A^2(\Omega)$ so that 
$f(a)=\langle f, K_a\rangle$ for any $f\in A^2(\Omega)$ and $a\in \Omega$.  
This special function $K_a$ is called the Bergman kernel.  When we 
need to be specific, we will use the notation $K^U$ to denote the 
Bergman kernel of $U$. One also has that
\[K_a(a)=\|K_a\|^2,\] 
$K_a(z)$ is holomorphic in $z$, and is conjugate holomorphic in $a$.

We also define the normalized Bergman kernel $k_a$ as
\[k_a(z)=\frac{K_a(z)}{\|K_a\|}=\frac{K_a(z)}{\sqrt{K_a(a)}}.\]   
The most fundamental questions are the boundedness and compactness of 
composition operators.  In the case $n = 1$, every composition operator is 
bounded as a consequence of the Littlewood Subordination theorem.  
MacCluer and Shapiro have characterized compactness of 
$C_{\varphi} $ on $ A^2(\mathbb{D})$ in \cite{MacCluerShapiro1986}.  
They proved that for $ \varphi : \mathbb{D} \rightarrow \mathbb{D}$ 
holomorphic, $C_{\varphi}$ is compact if and only if 
\[ \lim_{|z|\to 1^-} \frac{1 - |z|^2}{1 - |\varphi(z)|^2} = 0.\]  
Using the well known Julia-Caratheodory theorem, this result can be 
rephrased as $ C_{\varphi}$ is compact on $A^2(\mathbb{D})$ if and only if 
$\varphi$ does not have a finite angular derivative at any point on the 
boundary of $\mathbb{D}$. We recommend the book 
\cite{CowenMacCluerBook} for more information about 
composition operators.  We also mention related results in 
\cite{CuckovicZhao2004,GorkinMacCluer2004}. 

The situation is quite different for composition operators on domains in 
$\mathbb{C}^n$ for $n\geq 2$.  We refer the reader to the papers  
\cite{CuckovicZhao2007,HaKhoi2019,KooLi2014,KooLi2016,KooSmith2007}, 
and \cite{Li1995}.

In case of the polydisc $\mathbb{D}^n$, Jafari \cite{Jafari1990} characterized 
boundedness and compactness of these operators in terms of the following 
Carleson measure condition.  First we define the Carleson square 
$S^{\theta}_r$ as follows
\[S^{\theta}_r=\left\{\xi\in\disk:1-r<|\xi|<1,|\arg(\xi)-\theta|<r/2\right\}.\]

For any finite positive Borel measure $\lambda $ on $\mathbb{D}^n$,
we say that $\lambda $ is a Carleson measure if there exists $C>0$ such that 
\[\lambda\left(S_{r_1}^{\theta_1}\times\ldots\times S_{r_n}^{\theta_n}\right) 
	\leq CV\left(S_{r_1}^{\theta_1}\times\ldots\times S_{r_n}^{\theta_n}\right)\]
for all $0\leq \theta_j\leq2\pi, 0<r_j<1,$ and $1\leq j\leq n$. 
Furthermore, we say that $\lambda $ is a vanishing Carleson measure if 
\[\lim_{r_j\to 0^+}\sup_{(\theta_1,\ldots,\theta_n)\in [0,2\pi]^n} 
\frac{\lambda\left(S_{r_1}^{\theta_1}\times\cdots
	\times S_{r_n}^{\theta_n}\right)}{ 
V\left(S_{r_1}^{\theta_1}\times\cdots\times S_{r_n}^{\theta_n}\right)}=0\]  
for any $j$. 

The following theorem will be used in the proof of our main result. 
 
\begin{theoremA}\cite[Theorem 5, iii]{Jafari1990}\label{thmjaf}
Let $\varphi:\disk^n\to \disk^n$ be holomorphic, and $1<p<\infty$.  
Define pullback measure $V_{\varphi}$ to be
\[V_{\varphi}(E)=V(\varphi^{-1}(E))\] 
where $E\subset \overline{\disk^n}$.
Then $C_{\varphi}$ is compact on $A^p(\disk^n)$ if and only if
$V_{\varphi}$ is a vanishing Carleson measure. 
\end{theoremA}

It is well-known that in many cases boundedness of the composition operator 
is equivalent to a bounded Carleson measure condition and compactness is 
equivalent to a vanishing Carleson condition.  The papers \cite{Kosinski2023} 
and  \cite{Bayart2011} expound on these ideas. Here we also mention a closely 
related work in \cite{Choe1992}.  However, working with pull-back measures 
satisfying a Carleson measure condition is not an easy task.  As a result, in 
recent years, there has been a lot of work done on the boundedness of 
composition operators with symbols smooth up to the boundary acting 
on the Bergman spaces on the polydisc.  We  particularly mention 
\cite{Bayart2011,Kosinski2023, KooStessinZhu2008}, 
as well as \cite{StessinZhu2006}. 
 
We take a somewhat different approach to studying compactness 
of composition  operators on the Bergman space of the polydisc. 
Assuming more symbol regularity up to the closure of the polydisc 
allows us to characterize compactness in terms of the geometry 
of the image of the closure of the polydisc in the boundary.

We let $\|\cdot\|_E$ be the Euclidean norm and we define 
$d_{b\D}(z)=\inf\{\|z-w\|_E:w\in b\D\}$ to be the distance to $b\D$, 
the boundary of $\D$. We denote the unit circle by $\mathbb{T}$. 

We expect the situation on the polydisc to be much more complicated, 
as there are many more cases to consider. However, in case of the bidisc,  
we have the following result. 

\begin{theorem}\label{ThmBidisc}
Let $\varphi=(\varphi_1,\varphi_2):\mathbb{D}^2\to \mathbb{D}^2$ be a 
holomorphic self-map such that both $\varphi_1$ and $\varphi_2$ are Lipschitz 
on $\overline{\mathbb{D}^2}$.  Then   $C_{\varphi}$ is compact on 
$A^2(\mathbb{D}^2)$ if and only if  
$\varphi(\overline{\mathbb{D}^2})\cap \mathbb{T}^2=\emptyset$ and 
$\varphi(\overline{\mathbb{D}^2}\setminus \mathbb{T}^2)
\cap b\mathbb{D}^2=\emptyset$. 
\end{theorem}

In the example below, we show that without the Lipschitz 
condition, the forward direction in Theorem \ref{ThmBidisc} fails. 
However, the other direction still holds (see Corollary \ref{Cor2}).
\begin{example} 
The function $\xi+1$ maps the unit disc to $U$, the disc 
centered at 1 with radius 1. Then $\sqrt{\xi}$ maps $U$ into 
$\{\xi\in U: |Arg(\xi)|<\pi/4\}$. Let us define $\phi(\xi)=\sqrt{\xi+1}-1$. 
Then $\phi$ is a holomorphic self-map of the unit disc that is 
continuous but not Lipschitz on $\overline{\disk}$, $\phi'(-1)=\infty$, 
and $\phi(\overline{\disk})\cap \mathbb{T}=\{-1\}$. 

Next we will show that $C_{\phi}$ is compact on $A^2(\disk)$ by 
showing that the angular derivative does not exist at any point 
in $\mathbb{T}$.  Since $\phi(\overline{\disk})\cap \mathbb{T}=\{-1\}$, 
it suffices to show that the angular derivative is infinite at $\xi=-1$.
Define $\xi_k=-1+k^{-1}$.  Now we can compute
\[\frac{1-|\phi(\xi_k)|^2}{1-|\xi_k|^2} 
=\frac{\sqrt{k}\left(2-\frac{1}{\sqrt{k}}\right)}{2-\frac{1}{k}}\to \infty\]
as $k\to \infty$. Thus by \cite{MacCluerShapiro1986},
$C_{\phi}$ is compact on $A^2(\disk)$. Then by 
Theorem \ref{thmjaf} we have 
\[\lim_{r\to 0^+} \sup_{\theta\in[0,2\pi]}
\frac{V_{\phi}(S_r^{\theta})}{V(S_r^{\theta})}=0.\]

Next we define $\varphi(z_1,z_2)=(\phi(z_1),\phi(z_2))$. Then 
$\varphi$ is a holomorphic self-map of $\disk^2$ such that 
each component  is continuous on $\overline{\disk^2}$
but neither component is Lipschitz on $\overline{\disk^2}$.  
Furthermore, 
$\varphi(\overline{\disk^2})\cap \mathbb{T}^2\neq \emptyset$. 
Then 
\begin{align*}
\lim_{r_1,r_2\to 0^+} \sup_{(\theta_1,\theta_2)\in[0,2\pi]^2}
\frac{V_{\varphi}(S_{r_1}^{\theta_1}\times S_{r_2}^{\theta_2} )}{
	V(S_{r_1}^{\theta_1}\times S_{r_2}^{\theta_2})} 
= \left(\lim_{r_1\to 0^+} \sup_{\theta_1\in[0,2\pi]}
\frac{V_{\phi}(S_{r_1}^{\theta_1})}{V(S_{r_1}^{\theta_1})}\right) 
\left(\lim_{r_2\to 0^+} \sup_{\theta_2\in[0,2\pi]}
\frac{V_{\phi}(S_{r_2}^{\theta_2})}{V(S_{r_2}^{\theta_2})}\right)=0.
\end{align*} 
Finally, we use  Theorem \ref{thmjaf} and conclude that 
$C_{\varphi}$ is compact on $A^2(\disk^2)$. 
\end{example} 

In the next section, we will prove our main result, Theorem \ref{ThmBidisc}.  
In the last section, we will state a necessary condition for 
compactness of composition operators on $C^2$-smooth 
bounded pseudoconvex domains in $\C^n$. 

%%%%%%%%%%%%%%%%%%%%%%%%%%%%%%%%%%%%
\section*{Proof of Theorem \ref{ThmBidisc}}
In the proof of Theorem \ref{ThmBidisc} we will use the following proposition. 
A result with a similar flavor appeared in \cite[Theorem 4]{StessinZhu2006} 
for the weighted Bergman spaces on the polydisc. For general notions in 
several complex variables, such as pseudoconvexity and hyperconvexity, 
we refer the reader to the books 
\cite{ChenShawBook,JarnickiPflugBook2ndEd,KrantzBook,RangeBook}.

We note that even though \cite[Lemma 1]{RodriguezSahutoglu2024} is stated 
for bounded pseudoconvex domains with Lipschitz boundary 
(hence hyperconvex), the proof works for bounded hyperconvex 
domains as well. Therefore, we state the following slight 
generalization without a proof.

\begin{lemma}\label{LemWeakConv}
Let $\D$ be a bounded hyperconvex domain in $\C^n$ and $p\in b\D$. 
Then $k_z\to 0$ weakly in $A^2(\D)$ as $z\to p$. 
\end{lemma}	 

\begin{proposition}\label{Prop*} 
Let $\D$ be a bounded hyperconvex domain in $\C^n$  and  $\varphi=(\varphi_1,\ldots,\varphi_n):\D\to \D$ be a holomorphic 
self-map such that $\varphi_j\in C(\overline{\D})$ for all $j$.  
Assume that $C_{\varphi}$ is compact on $A^2(\D)$. Then 
\[\lim_{\D\ni z\to b\D}\frac{K_{\varphi(z)}(\varphi(z))}{K_z(z)}=0.\]
\end{proposition}

\begin{proof}
Assume that $C_{\varphi}$ is compact, $p\in b\D$, and $\{p_j\}\subset \D$  
so that $p_j\to p$ as $j\to\infty$. If $\varphi(p)\in \D$ then 
$K_{\varphi(p_j)}(\varphi(p_j))/K_{p_j}(p_j)\to 0$ because 
$K_{p_j}(p_j)\to \infty$ (see \cite[Theorem 12.4.4]{JarnickiPflugBook2ndEd})  
while $K_{\varphi(p_j)}(\varphi(p_j))$ stays bounded. On the other hand, if 
$\varphi(p)\in b\D$ then, by  Lemma \ref{LemWeakConv}, 
$k_{\varphi(p_j)}\to 0$ weakly as $j\to\infty$.  Then 
\[\|k_{\varphi(p_j)}\circ \varphi\| 
\geq \left|\left\langle k_{\varphi(p_j)}\circ \varphi, k_{p_j}\right\rangle\right|
=\left(\frac{K_{\varphi(p_j)}(\varphi(p_j))}{K_{p_j}(p_j)}\right)^{1/2}.\] 
Also compactness of $C_{\varphi}$ implies that 
$k_{\varphi(p_j)}\circ \varphi =C_{\varphi}k_{\varphi(p_j)}\to 0$ in 
$A^2(\D)$ as $j\to\infty$. 
Therefore,  the proof of the proposition is complete. 
\end{proof}

We don't know if converse of Proposition \ref{Prop*} is true, in general. 
However, at the end of the paper in Corollary \ref{Cor1}, we will show 
that the converse of Proposition \ref{Prop*} is true on the bidisc for 
Lipschitz symbols.

The following lemma is probably known to experts, but we provide a 
proof for the convenience of the reader. 
\begin{lemma}\label{Lem1} 
There exist $C>0$ such that 
\[V_{\varphi}(S^{\theta}_r)\leq C \|C_{\varphi}\|^2 V(S^{\theta}_r)\]
for any holomorphic self-map $\varphi:\disk\to\disk$ and 
$0<r\leq 1, 0\leq \theta\leq 2\pi$.
\end{lemma}
\begin{proof}
Let $0<r\leq 1$ and $a=(1-r)e^{i(\theta+r/2)}$. Then for 
$z=\rho e^{i\alpha} \in S^{\theta}_r$ we have   	
\begin{align*}
|1- a\zb|^2
=&\, (1-|a|\rho\cos(\theta+r/2-\alpha))^2+|a|^2\rho^2\sin^2(\theta+r/2-\alpha) \\
=&\,1+|a|^2\rho^2-2|a|\rho\cos(\theta+r/2-\alpha). 
\end{align*}
To get the maximum for $|1- a\zb|^2$ on $\overline{S^{\theta}_r}$ we need to 
choose $z_0\in \overline{S^{\theta}_r}$ so that $a\zb_0$ has the largest  
possible argument $r$ and smallest norm $(1-r)^2$.  Then we choose 
$z_0=(1-r)e^{i(\theta-r/2)}$. Hence, for $z \in S^{\theta}_r$ we have 
\[|1- a\zb|^2\leq |1- a\zb_0|^2=  1+|a|^4-2|a|^2\cos(1-|a|).\]
One can check that there exists $C_1>0$ such that 
\[0\leq  1+|a|^4-2|a|^2\cos(1-|a|)\leq C_1(1-|a|)^2\]
for $0\leq |a|\leq 1$. 
Hence, for $0<r\leq 1$ and $z\in S^{\theta}_r$, the normalized kernel 
has the following estimate
\[ |k_a(z)|=\frac{1-|a|^2}{|1- a\zb|^2} 
\geq \frac{1-|a|^2}{ C_1(1-|a|)^2}
\geq\frac{1}{ C_1(1-|a|)}=\frac{1}{ C_1r}.\]
Therefore, there exists $C_2>0$ independent of $a$ and $\theta$ such that 
\[ |k_a(z)|^2 \geq \frac{C_2}{(1-|a|)^2}=\frac{C_2}{r^2}\]
for $z\in S^{\theta}_r$ and $0<r\leq 1$. 
Next, we use the assumption that $C_{\varphi}$ is bounded. 
\[\frac{C_2V_{\varphi}(S^{\theta}_r)}{r^2}\leq \int_{S^{\theta}_r}|k_a(z)|^2dV_{\varphi} 
\leq \int_{\disk}|k_a(z)|^2dV_{\varphi} 
\leq  \|C_{\varphi}\|^2\int_{\disk}|k_a(z)|^2dV(z)=\|C_{\varphi}\|^2.\] 
Then, using the fact $V(S^{\theta}_r)= r^2-r^3/2$, we conclude that there exists 
a constant $C>0$ independent of $\varphi, r,$ and $\theta$ such that 
\[V_{\varphi}(S^{\theta}_r)\leq C\|C_{\varphi}\|^2 V(S^{\theta}_r)\]
for $0<r\leq 1$ and $0\leq\theta\leq 2\pi$. We note that $\|C_{\varphi}\|<\infty$ 
by Littlewood Subordination Theorem. 
\end{proof}

\begin{proof}[Proof of Theorem \ref{ThmBidisc}]
First, we assume that $C_{\varphi}$ is compact and both $\varphi_1$ and $\varphi_2$ 
are Lipschitz on $\overline{\mathbb{D}^2}$.  Furthermore, let us assume that 
$\varphi (\overline{\disk^2})\cap \mathbb{T}^2\neq \emptyset$. 
Let  $\varphi(p,q)\in \mathbb{T}^2$ for some $(p,q)\in b\mathbb{D}^2$.  
There are several cases to consider.  

First we consider the case where $(p,q)\in \mathbb{T}^2$.  That is, $|p|=1$ 
and $|q|=1$.  Without loss of generality, by multiplying by unimodular 
constants (which will not affect the compactness of $C_{\varphi}$), 
we may assume that $p=q=1$.  Let us define $p_k=p-1/k$.  Then 
\[d_{b\mathbb{D}}(p_k)=|p-p_k|=\frac{1}{k}.\] 
By Proposition \ref{Prop*}, we have
\[0=\lim_{k\to \infty}
\frac{K^{\mathbb{D}}_{\varphi_1(p_k,p_k)}(\varphi_1(p_k,p_k))
K^{\mathbb{D}}_{\varphi_2(p_k,p_k)}(\varphi_2(p_k,p_k))}{
\left(K^{\mathbb{D}}_{p_k}(p_k)\right)^2}.\]  
So, without loss of generality, we may assume that
\[0=\liminf_{k\to \infty}
\frac{K^{\mathbb{D}}_{\varphi_1(p_k,p_k)}(\varphi_1(p_k,p_k))}{
	K^{\mathbb{D}}_{p_k}(p_k)}
= \liminf_{k\to \infty}
\left(\frac{d_{b\mathbb{D}}(p_k)}{d_{b\mathbb{D}}(\varphi_1(p_k,p_k))}\right)^2.\]
Therefore, 
\begin{align}\label{EqnLimSup}
\limsup_{k\to \infty}
\frac{d_{b\mathbb{D}}(\varphi_1(p_k,p_k))}{d_{b\mathbb{D}}(p_k)}=\infty.
\end{align}
The function $\varphi_1$ is Lipschitz continuous on $\overline{\mathbb{D}^2}$ 
by assumption.  Hence for some $M>0$ we have
\[d_{b\mathbb{D}}\circ \varphi_1(p_k,p_k)
\leq |\varphi_1(p_k,p_k)-\varphi_1(p,p)|\leq M|p-p_k|.\]  
This is a contradiction to \eqref{EqnLimSup} as $|p-p_k|=d_{b\mathbb{D}}(p_k)$.   
Hence $\varphi(\mathbb{T}^2)\cap \mathbb{T}^2=\emptyset$. 

Now consider the case where $p\in b\mathbb{D}$ and $q\in \mathbb{D}$ 
where $\varphi(p,q)\in \mathbb{T}^2$. Without loss of generality, we 
assume that $p=1$. We define $(p_k,q_k)=(1-1/k,q)\in \mathbb{D}^2$ 
for $k\in \mathbb{N}$.  Then as before, we can write
\[\lim_{k\to \infty}\frac{d_{b\mathbb{D}}\circ 
\varphi_1(p_k,q)d_{b\mathbb{D}}\circ 
\varphi_2(p_k,q)}{d_{b\mathbb{D}}(p_k)d_{b\mathbb{D}}(q)}=\infty.\]  
Since
$d_{b\mathbb{D}}(q)>0$, we have
\[\lim_{k\to \infty}\frac{d_{b\mathbb{D}} 
	\circ \varphi_1(p_k,q)}{d_{b\mathbb{D}}(p_k)}=\infty.\]
Again using the fact that
$\varphi_1$ is Lipschitz on $\overline{\mathbb{D}^2}$, we 
have, for some $N>0$,
\[d_{b\mathbb{D}}\circ \varphi_1(p_k,q)
\leq |\varphi_1(p_k,q)- \varphi_1(p,q)| \leq N |p-p_k|.\]
As before, this is also a contradiction.  Therefore, we can conclude that
$\varphi(\overline{\mathbb{D}^2})\cap \mathbb{T}^2=\emptyset$.

Next we assume that $C_{\varphi}$ is compact and 
$\varphi(\overline{\mathbb{D}^2}\setminus \mathbb{T}^2)
\cap b\disk^2\neq\emptyset$. 
The fact that $\varphi$ is a self-map of $\disk^2$ implies that 
$\varphi(b\mathbb{D}^2\setminus \mathbb{T}^2)\cap b\disk^2\neq \emptyset$. 
Then there exists $(p,q)\in b\mathbb{D}^2\setminus \mathbb{T}^2$ 
so that $\varphi(p,q)\in b\mathbb{D}^2$. Without loss of generality, 
we may assume that $p\in b\mathbb{D}$ and $q\in \mathbb{D}$. 
Also, by the first part of this proof, we may assume that
$\varphi_1(p,q)\in b\mathbb{D}$ and $\varphi_2(p,q)\in \mathbb{D}$.
Then by Proposition \ref{Prop*}, we have
\[\lim_{\mathbb{D}^2\ni (z_1,z_2)\to (p,q)}
\frac{K^{\mathbb{D}^2}_{\varphi(z_1,z_2)}(\varphi(z_1,z_2))}{
	K^{\mathbb{D}^2}_{(z_1,z_2)}(z_1,z_2)}=0.\]
Thus we have
\[\lim_{\mathbb{D}\ni (z_1,z_2)\to (p,q)}
\frac{d_{b\mathbb{D}}(\varphi_1(z_1,z_2))d_{b\mathbb{D}} 
(\varphi_2(z_1,z_2))}{d_{b\mathbb{D}}(z_1)d_{b\mathbb{D}}(z_2)}
=\infty.\]
By assumption, 
\[\frac{d_{b\mathbb{D}}(\varphi_2(z_1,z_2))}{d_{b\mathbb{D}}(z_2)}\] 
is bounded for $(z_1,z_2)$ near $(p,q)$. Fixing $z_2=q$, we can conclude that
\[\lim_{\mathbb{D}\ni z_1\to p}
\frac{d_{b\mathbb{D}}(\varphi_1(z_1,q))}{d_{b\mathbb{D}}(z_1)}=\infty.\]
Now  we let $\nu$ be the unit outward normal vector at $p$ and 
define $p_k=p-\frac{\nu}{k}$.  It is clear that 
$d_{b\mathbb{D}}(p_k)=|p_k-p|=\frac{1}{k}$.  Then we have
\[\lim_{k\to \infty}\frac{d_{b\mathbb{D}}(\varphi_1(p_k,q))}{|p_k-p|}=\infty.\]  
This is a contradiction since we have the following estimate.
\[d_{b\mathbb{D}}\circ \varphi_1(p_k,q)\leq |\varphi_1(p_k, q)-\varphi_1(p,q)|
\leq N|p_k-p|=N\,d_{b\mathbb{D}}(p_k).\] 
Hence, 
$\varphi(\overline{\mathbb{D}^2}\setminus \mathbb{T}^2) 
	\cap b\mathbb{D}^2=\emptyset$. 
Therefore, we showed that if $C_{\varphi}$ is compact then 
$\varphi (\overline{\disk^2})\cap \mathbb{T}^2= \emptyset$ and
$\varphi(\overline{\mathbb{D}^2}\setminus \mathbb{T}^2) 
	\cap b\mathbb{D}^2=\emptyset$. 

Next to prove the converse we will assume that 
$\varphi (\overline{\disk^2})\cap \mathbb{T}^2= \emptyset$ and
$\varphi(\overline{\mathbb{D}^2}\setminus \mathbb{T}^2) 
	\cap b\mathbb{D}^2=\emptyset$. 
Then we will  use Jafari's condition, Theorem \ref{thmjaf}, 
to prove that $C_{\varphi}$ is compact. First we note that   
$\varphi(\overline{\mathbb{D}^2}\setminus \mathbb{T}^2) 
	\subset \mathbb{D}^2$  and hence 
$\varphi_j(\overline{\mathbb{D}^2}\setminus \mathbb{T}^2) 
	\subset \mathbb{D}$ for $j=1,2$. 
 
If $\varphi(\overline{\disk^2})\subset \disk^2$ then, using Cauchy's formula, 
one can show that $C_{\varphi}$ is compact. Without loss of generality, let 
$\Gamma_1=\varphi_1^{-1}(b\mathbb{D})\subseteq \varphi^{-1}(b\disk^2) 
	\subseteq \mathbb{T}^2$ be a non-empty set and denote 
\[\Gamma_1(\ep_1)
=\left\{z\in \overline{\mathbb{D}^2}:d_{\Gamma_1}(z)<\ep_1\right\}.\] 
Since $\varphi_1(\overline{\mathbb{D}^2}\setminus \Gamma_1(\ep_1))$ is 
a compact subset of $\mathbb{D}$ for all $\ep_1>0$,  there exists 
$\ep_2>0$ such that 
$U_{\ep_2}^{\theta_1}=\varphi_1^{-1}(S_{\ep_2}^{\theta_1})
\subset \Gamma_1(\ep_1)$ 
for  $0\leq \theta_1\leq 2\pi$.  Then we conclude that 
\[\eta(\ep,\theta_1)=\inf\left\{|z_1|:(z_1,z_2)\in U_{\ep}^{\theta_1}\right\}
\geq 1-\ep_1\]
for $0< \ep\leq \ep_2$ and all $\theta_1\in[0,2\pi]$. Namely, for 
$\ep_1>0$ there exists $\ep_2>0$ such that 
$1-\ep_1\leq \eta(\ep,\theta_1)\leq 1$ for  $0< \ep\leq \ep_2$. 
Hence $\eta(\ep,\theta_1)\to 1$ uniformly in $\theta_1$ as  $\ep\to 0^+$.

Since we assumed that 
$\varphi(\overline{\mathbb{D}^2})\cap \mathbb{T}^2=\emptyset$, 
it suffices to show Jafari's compactness condition is satisfied. We will 
apply Theorem \ref{thmjaf} to $C_{\varphi}$ for $r_1\to 0^+$ while 
$r_2>0$ fixed. If both $r_1$ and $r_2$ go to $0^+$, Jafari's condition 
is satisfied vacuously since 
$\varphi^{-1}(S_{r_1}^{\theta_1}\times S_{r_2}^{\theta_2})=\emptyset$ 
for all $r_1>0$ and all $r_2>0$ sufficiently small. 

We note that 
\[\varphi^{-1}(S_{r_1}^{\theta_1}\times S_{r_2}^{\theta_2})
= \varphi_1^{-1}(S_{r_1}^{\theta_1})\cap \varphi_2^{-1}(S_{r_2}^{\theta_2}).\]
For $z_1\in\overline{\disk}$  we denote 
$U_{r_1z_1}^{\theta_1}=\{z_2\in \disk:(z_1,z_2)\in U_{r_1}^{\theta_1}\}$,  
where $U_{r_1}^{\theta_1}=\varphi_1^{-1}(S_{r_1}^{\theta_1})$. 
In the last inequality below, we use Lemma \ref{Lem1} to get a universal constant 
$C>0$ independent $\varphi_1, z_1,r_1$ and $\theta_1$.  
\begin{align*}
V(\varphi^{-1}(S^{\theta_1}_{r_1}\times S^{\theta_2}_{r_2})) 
=&\int_{\varphi^{-1}(S^{\theta_1}_{r_1}\times S^{\theta_2}_{r_2})}dV (z_1,z_2)\\
\leq & \int_{\varphi_1^{-1}(S^{\theta_1}_{r_1})}dV (z_1,z_2)\\
=&\int_{U_{r_1}^{\theta_1}}dV(z_1,z_2) \\
= &\, \int_{\eta(r_1,\theta_1)\leq |z_1|\leq 1} 
	\int_{z_2\in U_{r_1z_1}^{\theta_1}}dV(z_2)dV(z_1)\\
=&\, \int_{\eta(r_1,\theta_1)\leq |z_1|\leq 1}
	V(U_{r_1z_1}^{\theta_1})dV(z_1)\\
\leq &\, \int_{\eta(r_1,\theta_1)\leq |z_1|\leq 1}C 
\|C_{\varphi_1(z_1,\cdot)}\|^2  V(S_{r_1}^{\theta_1})  dV(z_1).
\end{align*}
We note that $\varphi_1(\cdot,0)$ maps $\overline{\disk}$ into 
$\disk$. Hence, there exists $0\leq \lambda<1$ such that 
$|\varphi_1(z_1,0)|\leq \lambda$ for all $|z_1|\leq 1$. 
By \cite[Theorem 11.6]{ZhuBook},   
\[\|C_{\varphi_1(z_1,\cdot)}\| 
\leq \frac{1+|\varphi_1(z_1,0)|}{1-|\varphi_1(z_1,0)|}
\leq \frac{1+\lambda}{1-\lambda}\] 
for all $|z_1|\leq 1$. Then we get 
\begin{align*}
V(\varphi^{-1}(S^{\theta_1}_{r_1}\times S^{\theta_2}_{r_2})) 
\leq& \,  \int_{\eta(r_1,\theta_1)\leq |z_1|\leq 1}C 
\|C_{\varphi_1(z_1,\cdot)}\|^2  V(S_{r_1}^{\theta_1})  dV(z_1)\\
\leq&\, \pi C(1-\eta(r_1,\theta_1)^2) 
 V(S_{r_1}^{\theta_1})\left(\frac{1+\lambda}{1-\lambda}\right)^2. 
\end{align*} 
Since $r_2>0$ is fixed, there exists $D>0$ independent of $r_1,\theta_1$, 
and $\theta_2$ such that
\[\frac{V_{\varphi}(S^{\theta_1}_{r_1}\times S^{\theta_2}_{r_2})}{ 
V(S_{r_1}^{\theta_1}\times S_{r_2}^{\theta_2})} 
\leq D(1-\eta(r_1,\theta_1)^2) \left(\frac{1+\lambda}{1-\lambda}\right)^2.\]
Since $\eta(r_1,\theta_1)\to 1$ uniformly in $\theta_1$ as $r_1\to 0^+$, 
$V_{\varphi}$ is a vanishing Carleson measure and 
we conclude that $C_{\varphi}$ is compact because Jafari's condition in  
Theorem \ref{thmjaf} is satisfied.
\end{proof}

The first part of the proof of Theorem \ref{ThmBidisc} shows that if 
$K_{\varphi(z)}(\varphi(z))/K_z(z)\to 0$ as $z\to b\disk^2$ then 
$\varphi(\overline{\mathbb{D}^2})\cap \mathbb{T}^2=\emptyset$ and 
$\varphi(\overline{\mathbb{D}^2}\setminus \mathbb{T}^2)
\cap b\mathbb{D}^2=\emptyset$. Then Theorem \ref{ThmBidisc} implies 
that $C_{\varphi}$ is compact on $A^2(\mathbb{D}^2)$.  
Hence we have the following corollary. 

\begin{corollary}\label{Cor1} 
Let $\varphi=(\varphi_1,\varphi_2):\mathbb{D}^2\to \mathbb{D}^2$ be a 
holomorphic self-map such that both $\varphi_1$ and $\varphi_2$ are 
Lipschitz  on $\overline{\mathbb{D}^2}$.  Then $C_{\varphi}$ is 
compact on  $A^2(\mathbb{D}^2)$ if and only if
\[\lim_{\disk^2\ni z\to b\disk^2}\frac{K_{\varphi(z)}(\varphi(z))}{K_z(z)}=0.\]
\end{corollary} 

In the second part of the proof of Theorem \ref{ThmBidisc}, we observe that 
continuity of the symbol on $\overline{\mathbb{D}^2}$ is sufficient. 
Hence we have the following corollary. 

\begin{corollary}\label{Cor2}
Let $\varphi=(\varphi_1,\varphi_2):\mathbb{D}^2\to \mathbb{D}^2$ be a 
holomorphic self-map such that both $\varphi_1$ and $\varphi_2$ are 
continuous on $\overline{\mathbb{D}^2}$.  Assume that 
 $\varphi(\overline{\mathbb{D}^2})\cap \mathbb{T}^2=\emptyset$ and 
$\varphi(\overline{\mathbb{D}^2}\setminus \mathbb{T}^2)
\cap b\mathbb{D}^2=\emptyset$. Then   $C_{\varphi}$ is compact on 
$A^2(\mathbb{D}^2)$. 
\end{corollary}

Below we give a simple example satisfying the conditions of 
Theorem \ref{ThmBidisc}. 

\begin{example}\label{ExBidisc} 
Let $\varphi(z_1,z_2)=(z_1/2, z_1z_2)$ be a self-map of the bidisc. One can 
easily see that $\varphi(z_1,z_2)\in b\disk^2$ if and only if $|z_1|=|z_2|=1$. 
Furthermore, $\varphi(\overline{\disk^2})\cap \mathbb{T}^2=\emptyset$. 
Therefore, by Theorem \ref{ThmBidisc}, $C_{\varphi}$ is compact. 
\end{example}

%%%%%%%%%%%%%%%%%%%%%%%%%%%%%%%%%%
\section*{Smooth Pseudoconvex Domains}
In this section, we state a simple theorem on $C^2$-smooth bounded 
pseudoconvex domains in $\C^n$. We expect that heavier several 
complex variables techniques would be needed to get better results. 
The following theorem is a consequence of 
\cite[Theorem 3.5.1]{Hormander1965}. 

\begin{theoremA}[H\"ormander]\label{ThmHormanderEst}
Let $\D$ be a $C^2$-smooth bounded pseudoconvex domain in $\C^n$. 
Assume that $p\in b\D$ is a strongly pseudoconvex point. 
Then there exist an open neighborhood $U$ of $p$ and  $C>0$ such that 
\[\frac{1}{C(d_{b\D}(z))^{n+1}}\leq K_z(z)\leq  \frac{C}{(d_{b\D}(z))^{n+1}}\]
for $z\in \D\cap U$. 
\end{theoremA}

\begin{corollary}\label{CorGenEstimate} 
Let $\D$ be a $C^2$-smooth bounded domain in $\C^n$. 
Then there exists $C>0$ such that 
\[K_z(z)\leq  \frac{C}{(d_{b\D}(z))^{n+1}}\]
for $z\in \D$. 
\end{corollary}
\begin{proof}
First we note that if $U\subset \D$ is a domain then $K^{\D}_z(z)\leq K^U_z(z)$. 
Second, since $\D$ is $C^2$-smooth and bounded, there exists a ball 
$B$ centered at the origin such that for any $w\in b\D$ there exists 
$\widetilde{w}\in \D$ such that $\widetilde{w}+B\subset \D$ 
and $\{w\}=\overline{(\widetilde{w}+B)}\cap b\D$. 
Then there exists $C>0$ such that for $z\in \D$ sufficiently close to 
$b\D$, there exists $\widetilde{w}\in \D$ such that 
$z\in \widetilde{w}+B\subset \D$ and 
$d_{b(\widetilde{w}+B)}(z)=d_{b\D}(z)$. Then 
\[K_z(z)\leq  K^{\widetilde{w}+B}_z(z)
\leq \frac{C}{(d_{b(\widetilde{w}+B)}(z))^{n+1}}
=\frac{C}{(d_{b\D}(z))^{n+1}}.\] 
Finally, since $K_z(z)$ is bounded on compact  subsets of $\D$, 
we can choose $C>0$  so that 
\[K_z(z)\leq  \frac{C}{(d_{b\D}(z))^{n+1}}\]
$z\in \D$. 
\end{proof}

We  define $\mathcal{S}_{\D}$ to be the set of strongly pseudoconvex
points in $b\D$. The theorem below seem to suggest that, in case of 
$C^2$-smooth bounded pseudoconvex domains in $\C^n$, 
strongly pseudoconvex points should play the role of the 
distinguished boundary in the bidisc.

\begin{theorem}\label{ThmStrongPsdxPoints}
Let $\D$ be a  $C^2$-smooth bounded pseudoconvex 
domain in $\C^n$ and  $\varphi=(\varphi_1,\ldots,\varphi_n):\D\to \D$ 
be a holomorphic self-map such that $\varphi_j$ is Lipschitz on 
$\overline{\D}$ for all $j$. Assume that $C_{\varphi}$ is compact 
on $A^2(\D)$. Then $\varphi(b\D)\cap \mathcal{S}_{\D}=\emptyset$.	
\end{theorem}

\begin{proof}
For the sake of obtaining a contradiction, assume that 
$C_{\varphi}$ is compact on $A^2(\D)$ and $\varphi(p)\in b\D$ is a 
strongly pseudoconvex point for some $p\in b\D$.  
Then $b\D$ is strongly pseudoconvex near $\varphi(p)$. Then, using 
Theorem \ref{ThmHormanderEst} and Corollary \ref{CorGenEstimate},
there exists $C>0$ such that  for $z\in \D$ sufficiently close to $p$ we have  
\[\frac{K_z(z)}{K_{\varphi(z)}(\varphi(z))} 
\leq C\left( \frac{d_{b\D}(\varphi(z))}{d_{b\D}(z)}\right)^{n+1}.\]
Then we have, by Proposition \ref{Prop*},
\[\lim_{\D\ni z\to p} \frac{K_z(z)}{K_{\varphi(z)}(\varphi(z))}=\infty.\]
Hence 
\[\lim_{\D\ni z\to p}\frac{d_{b\D}(\varphi(z))}{d_{b\D}(z)}=\infty.\]
This contradicts the Lipschitz property of $\varphi_j$ on $\overline{\D}$ 
for some $j$.
\end{proof}

We finish the paper by constructing an example to show that the 
converse of Theorem \ref{ThmStrongPsdxPoints} is false.

\begin{figure}[h] %b:bottom; h:hold; t:top
\scalebox{1.1}{
\begin{tikzpicture}
	\draw[thick,->] (0,0) -- (2.5,0) node[anchor=north west] {};
	\draw[thick,->] (0,0) -- (0,2.5) node[anchor=south east] {};
	\node[draw=white] at (0,3) {$|z_2|$};
	\draw[dashed] (0,2) -- (1,2);
	\node[draw=white] at (1,1) {$\Omega$};
	\draw[dashed] (2,0) -- (2,1);
	\draw[dashed] (2,1) arc (0:90:1cm);
	\node[draw=white] at (3,0) {$|z_1|$};
\end{tikzpicture}
}
\caption{The domain in Example \ref{Example2}}
\label{Fig}
\end{figure} 

\begin{example}\label{Example2}
Let us define $\lambda(t)= 0$ for $t\leq 0,$  
$\lambda(t)= e^{-1/t}$ for $t>0$, and 
\[\rho(z_1,z_2)=\lambda\left(\frac{4|z_1|^2-1}{12}\right)
+\lambda\left(\frac{4|z_2|^2-1}{12}\right)-\lambda(1/4).\]
Then we define 
\[ \D=\left\{(z_1,z_2)\in \C^2: \rho(z_1,z_2)<0 \right\}.\]
One can check that $\lambda$ is a convex function on $(-\infty,1/2)$ 
and strictly convex on $(0,1/2)$. Then $\D$ is a $C^{\infty}$-smooth 
bounded convex and Reinhardt domain. We note that $\D$ is contained 
in $\mathbb{D}^2$ and is obtained by smoothing out the distinguished 
boundary $\mathbb{T}^2$ of the bidisc (see Figure \ref{Fig}).   

Let us denote  
\[W_{\D}=\left\{(z_1,z_2)\in \C^2:|z_1|<\frac{1}{2},|z_2|=1\right\}
\cup \left\{(z_1,z_2)\in \C^2:|z_1|=1,|z_2|<\frac{1}{2}\right\}.\] 
Since $\lambda$ is strictly convex on $(0,1/2)$, one can see 
that $\mathcal{S}_{\D}= b\D\setminus\overline{W}_{\D}$.
Let us define $\varphi(z_1,z_2)=\left(\frac{z_1}{2},z_2\right)$. Then  
\[\varphi(b\D)\cap b\D 
\subset \left\{(z_1,z_2)\in \C^2:|z_1|<\frac{1}{2},|z_2|=1\right\} 
\subset W_{\D}.\] 
Therefore, $\varphi(b\D)\cap \mathcal{S}_{\D}=\emptyset$. 

However, $C_{\varphi}$ is not compact. This can be seen as follows. 
Let $p_j=(0, 1-1/j)$. Then 	
\begin{align*}
\frac{K_{p_j}(p_j)}{K_{\varphi(p_j)}(\varphi(p_j))}
	=\frac{K_{p_j}(p_j)}{K_{p_j}(p_j)}=1.
\end{align*}
Hence, by Proposition \ref{Prop*}, we conclude that $C_{\varphi}$ is 
not compact. 
\end{example}

%%%%%%%%%%%%%%%%%%%%%%%%%%
\section*{Acknowledgment}
We are in debt to Trieu Le for noticing a mistake in an earlier version of 
the paper and that the proof of Theorem \ref{ThmBidisc} implies 
Corollary \ref{Cor1}. 
We thank the referee for reading the manuscript carefully 
and suggesting an improvement to Proposition \ref{Prop*}. 
%%%%%%%%%%%%%%%%%%%%%%

\end{document}